\documentclass[a4paper,12pt]{article}
\usepackage{amscd, color}
\usepackage{amsmath,amsfonts,amssymb,amscd}
\usepackage{indentfirst,graphicx,epsfig}
\usepackage{graphicx,psfrag}
\input{epsf}

\newtheorem{df}{Definition}
\newtheorem{thm}{Theorem}
\newtheorem{problem}{Problem}
\newtheorem{conjecture}{Conjecture}
\newtheorem{lem}{Lemma}

\newtheorem{claim}{Claim}
\newtheorem{Constrution}{Construction}

\newenvironment {proof} {\noindent{\em Proof.}}{\hspace*{\fill}$\Box$\par\vspace{4mm}}

\title{\bf On extremal problems on multigraphs }
\author{
	\small  Ran Gu$^1$, 
	Shuaichao Wang$^2$\\
	\small $^1$College of Science, Hohai University,\\
	\small Nanjing, Jiangsu Province 210098,
	P.R. China\\
	\small $^2$Center for Combinatorics and LPMC \\
	\small Nankai University, Tianjin 300071, China \\
	\small Emails: rangu@hhu.edu.cn;Wsc17746316863@163.com
	\\
	\date{}}

\makeatletter

\newcommand{\Rmnum}[1]{\expandafter@slowromancap\romannumeral #1@}
\makeatother
\begin{document}
	\maketitle
	\begin{abstract}
		An $(n,s,q)$-graph is an $n$-vertex multigraph in which every $s$-set of vertices spans at most $q$ edges. Erd\H{o}s initiated the study  of maximum number of edges of $(n,s,q)$-graphs, and the extremal problem on multigraphs has been considered since the 1990s.  The problem of determining the maximum product of the edge multiplicities in $(n,s,q)$-graphs was posed by Mubayi and Terry in  2019. Recently,  Day, Falgas-Ravry and Treglown settled a conjecture of Mubayi and Terry on the case $(s,q)=(4, 6a + 3)$ of the problem (for $a \ge 2$), and they gave a general lower
        bound construction for the extremal problem for many pairs $(s, q)$,
        which they conjectured is asymptotically best possible. Their conjecture was confirmed exactly or asymptotically for some specific cases.
      In this paper, we consider the case that $(s,q)=(5,\binom{5}{2}a+4)$ and $d=2$ of their conjecture,  partially solve an open problem raised by Day, Falgas-Ravry and Treglown.  We also show that  the conjecture fails for $n=6$, which indicates for the case  that  $(s,q)=(5,\binom{5}{2}a+4)$ and $d=2$,  $n$ needs to be sufficiently large for the conjecture to hold.
        \\[2mm]

	\noindent\textbf{Keywords:}  Multigraphs; Tur\'an problems; Extremal graphs\\
		[2mm] \textbf{AMS Subject Classification (2020):} 05C35, 05C22
	\end{abstract}
	
	\section{Introduction}
	In 1963, Erd\H{o}s \cite{Erdos1,Erdos2} posed the extremal question on  $ex(n,s,q)$, the maximum number of edges in an $n$-vertex graph in which  every $s$-set of vertices spans at most $q$ edges, where $q$ is an integer satisfying that $0\leq q \leq \binom{n}{2}$. In the 1990s, Bondy and Tuza \cite{BT} and Kuchenbrod \cite{K} raised an analogous extremal problem on multigraphs.
	A multigraph is a pair ($V,w$), where $V$ is a vertex set and $w$ is a function $w:\binom{V}{2} \rightarrow \mathbb{Z}_{\geq0}.$
	\begin{df}
		Given integers $s\geq 2$ and $q\geq 0$, we say a multigraph $G=(V,w)$ is an $(s,q)$-graph if every $s$-set of vertices in $V$ spans at most $q$ edges; i.e. ${\sum}_{xy\in \binom{X}{2}}w(xy)\leq q$ for every $X\in \binom{V}{s}$. An $(n,s,q)$-graph is an $n$-vertex $(s,q)$-graph. We write $\mathcal{F}(n,s,q)$ for the set of all $(n,s,q)$-graphs with vertex set $[n]$ $:=$ $\{1,...,n\}$.
	\end{df}
\par
     Bondy and Tuza \cite{BT}, Kuchenbrod \cite{K} and F\"{u}redi and K\"{u}ndgen \cite{FK} studied  the problem of the maximum of the sum of the edge multiplicities in an ($n,s,q$)-graph. Especially, F\"{u}redi and K\"{u}ndgen \cite{FK} obtained an asymptotically tight upper bound $m\binom{n}{2}+O(n)$ for the maximum   of edges in $(n,s,q)$-graphs, where  $m=m(s,q)$ is an explicit constant. Recently, Mubayi and Terry \cite{MT1,MT2} introduced a version of  multiplicity product of the problem as follows.
     \begin{df}
     	Given a multigraph $G=(V,w)$, we define
     	$$P(G) := \prod_{xy\in \binom{V}{2}}w(xy),$$
     	$$ex_\Pi (n,s,q) := \max\{P(G):G\in \mathcal{F}(n,s,q)\},$$
     	$$ex_\Pi (s,q) := \lim_{n\to +\infty}(ex_\Pi (n,s,q))^{\binom{n}{2}^{-1}}.$$
     \end{df}
 \par
     Mubayi and Terry  showed in \cite [Theorem 2.2]{MT1}, that for $q\geq \binom{s}{2}$,
    $$\left|\mathcal{F}(n,s,q-\binom{s}{2})
    \right|=ex_\Pi (s,q)^{\binom{n}{2}+o(n^2)}.$$
   As  Mubayi and Terry pointed out estimating the size of the multigraph family $\mathcal {F}\left(n, s, q -\binom{s}{2}\right)$ is equivalent to the Tur\'{a}n-type
extremal problem of determining $ex_{\Pi}(n, s, q)$. Therefore, Mubayi and Terry raised the general problem of determining $ex_\Pi (n,s,q)$ as below.
	\begin{problem} \cite{MT1,MT2}
		Given positive integers $s\geq 2$ and $q$, determine $ex_\Pi (n,s,q).$
	\end{problem}
\par
      Mubayi and Terry in \cite{MT1} proved that
      $$ex_\Pi (n,4,15)=2^{\gamma n^2+O(n)}$$
      where $\gamma$ is defined by
      $$\gamma:=\frac{\beta ^2}{2}+\beta(1-\beta)\frac{\log3}{\log2}~~\text{where}~~\beta:=\frac{\log3}{2\log3-\log2}.$$ This is the first example of a `fairly natural extremal graph problem' whose asymptotic answer is given by an explicitly defined transcendental number. In response to a question of Alon which asked whether this transcendental behaviour is an isolated case,
       Mubayi and Terry \cite{MT1} made a conjecture on the value of $ex_\Pi (n,4,\binom{4}{2}a+3)$. In \cite{DFT}, Day, Falgas-Ravry and Treglown resolved their conjecture fully,
       consequently providing infinitely many examples on Alon's question.

   Mubayi and Terry \cite{MT2} exactly or asymptotically determined $ex_\Pi(n, s, q)$ for pairs
$(s, q)$ where $a\binom{s}{2}-\frac{s}{2}\le q \le a\binom{s}{2}+ s - 2$ for some $a \in  \mathbb{N}$. Day, Falgas-Ravry and Treglown \cite{DFT} investigated $ex_\Pi(n, s, q)$ for a further range of values of $(s, q)$. In particular, they gave a potential extremal structure related to $ex_\Pi(n, s, q)$.  Their constructions  may be seen as a class of multigraphs analogues of the well known Tur\'{a}n graphs.

\begin{Constrution}\cite{DFT}
	Let $a,r\in \mathbb{N}$ and $d\in \left[0,a-1\right]$. Given $n\in \mathbb{N}$, let $\mathcal{T}_{r,d}(a,n)$ denote the collection of multigraphs $G$ on $\left[n\right]$ for which $V(G)$ can be partitioned into $r$ parts $V_0,...,V_{r-1}$ such that:\\
	(i) all edges in $G[V_0]$ have multiplicity $a-d$;\\
	(ii) for all $i\in [r-1]$, all edges in $G[V_i]$ have multiplicity $a$;\\
	(iii) all other edges of $G$ have multiplicity $a+1$.\\
\newline
	Given $G\in \mathcal{T}_{r,d}(a,n)$, we refer to $\cup_{i=0}^{r-1}V_i$ as the canonical partition of $G$.
\end{Constrution}
\par
    Fig. \ref{dhgraph} is an example of what these graphs look like when $r=4$. We write
    $${\Sigma}_{r,d}(a,n):=\max\left\lbrace  e(G):G\in \mathcal{T}_{r,d}(a,n) \right\rbrace  ,$$
    and
    $$\Pi_{r,d}(a,n):=\max\left\lbrace P(G):G\in \mathcal{T}_{r,d}(a,n)\right\rbrace .$$
\begin{figure}[http]
	\centering
	\scalebox{0.2}{\includegraphics{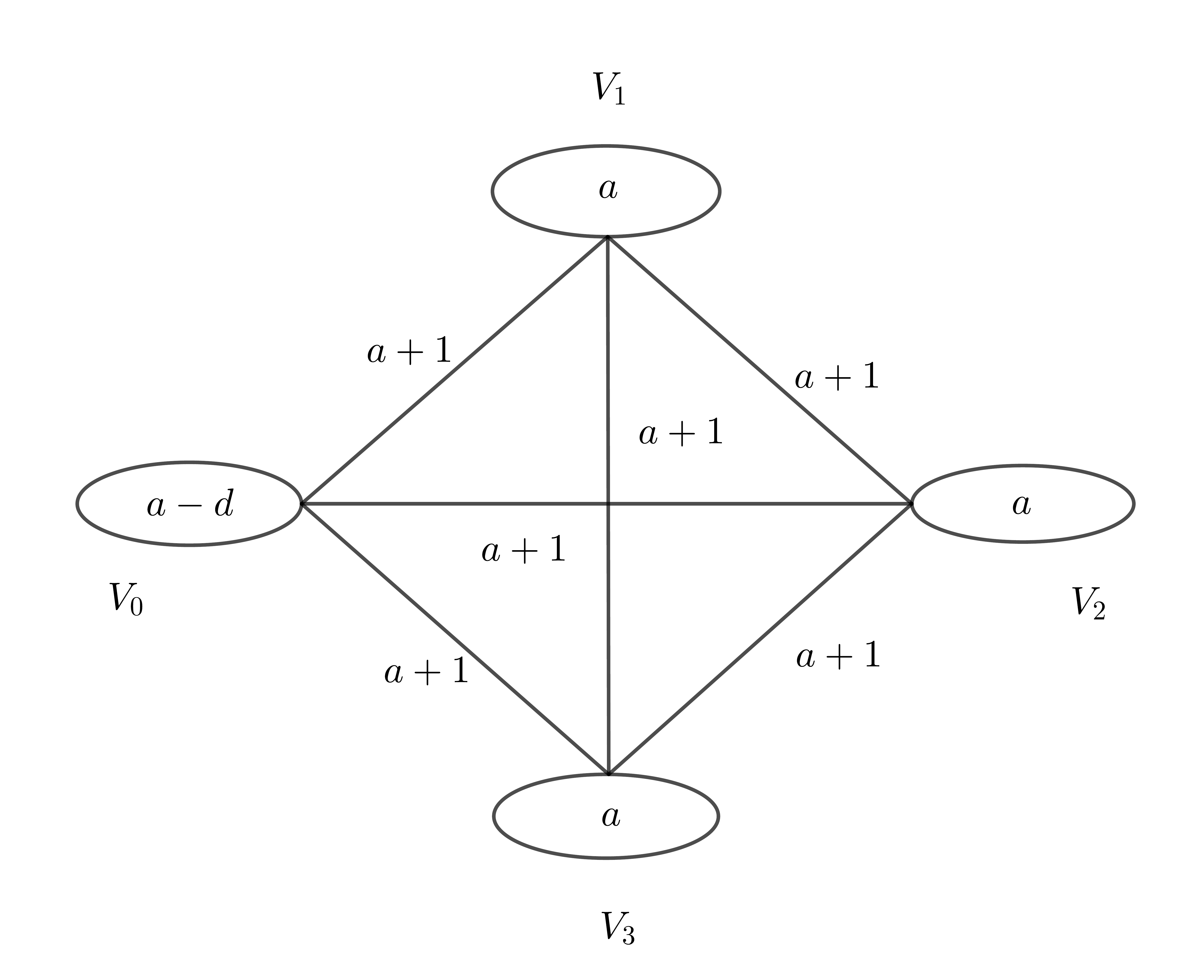}}
	\caption{An example of the structure of graphs in $\mathcal{T}_{4,d}(a,n).$}
	\label{dhgraph}
	
\end{figure}
\par
    In \cite{DFT},  Day, Falgas-Ravry and Treglown raised a new conjecture as below.
	\begin{conjecture}\label{DFTc}
	For all integers $a,r,s,d$ with $a,r\geq 1$, $d\in \left[0,a-1 \right]$,$s\geq(r-1)(d+1)+2$ and all $n$ sufficiently large,
\begin{equation}\label{conj}
 ex_\Pi (n,s,{\Sigma}_{r,d}(a,s))=\Pi_{r,d}(a,n).
\end{equation}
\end{conjecture}
\par

As the definition of ${\Sigma}_{r,d}(a,n)$, we know $ex_\Pi (n,s,{\Sigma}_{r,d}(a,s))\geq \Pi_{r,d}(a,n)$.
 Conjecture \ref{DFTc} roughly states that for any $d\in \left[0,a-1 \right]$ and other conditions of $n,s,a,r$, when we take $q={\Sigma}_{r,d}(a,s)$, it is a graph $G$ from $\mathcal{T}_{r,d}(a,n)$ maximises the edge-product $P(G)$ amongst all$(n,s,q)$-graphs.
  Mubayi and Terry \cite{MT2} proved that for all $r$ such that $\frac{s}{2} \le r \le s - 1$
and $n \ge s$, $ex_\Pi(n, s, \Sigma_{r,0}(a, n)) = \Pi_{r,0}(a, n)$,
which partly confirms  Conjecture \ref{DFTc} for $d = 0$ . In \cite{DFT}, Day, Falgas-Ravry and Treglown  completed the proof of the $d = 0$ case of Conjecture \ref{DFTc}. Combine the results in \cite{F} and \cite{DFT}, Conjecture \ref{DFTc} asymptotically holds for  $d=1$ and sufficiently large $a$.

  Note that the known results on Conjecture \ref{DFTc} are considering the cases $d=0$ or $d=1$.   In this paper, we consider \eqref{conj} for the case of $d=2$ and the special pair $(s,q)=(5,\binom{5}{2}a+4)$ which is mentioned as an open problem by  Day, Falgas-Ravry and Treglown in \cite{DFT}. We obtain the results as follows.
  \begin{thm}\label{MS1}
    For $(s,q)=(5,\binom{5}{2}a+4)$, we have
  	\begin{enumerate}
  			\item[(i)]
  		 For $n\geq 8$ and  $a\rightarrow \infty$, $ex_\Pi (n,5,\binom{5}{2}a+4)\rightarrow \Pi_{2,2}(a,n)$.
  		\item[(ii)]
  	    For $n=7$ and $a\ge 3$, we have
  		$$(a-2)(a+1)^{10}a^{10}\leq ex_\Pi (7,5,\binom{5}{2}a+4)<(a-1)a^{11}(a+1)^9. $$ Particularly, $ex_\Pi (7,5,\binom{5}{2}a+4)\rightarrow\Pi_{2,2}(a,7)$ as $a\rightarrow \infty$.
  		\item[(iii)]
  	 	 For $n=6$ and $a\ge 3$, $ex_\Pi (6,5,\binom{5}{2}a+4)=a^9(a+1)^6>\Pi_{2,2}(a,6).$ And
$$\Pi_{2,2}(a,6)=
\begin{cases}
	(a+1)^5a^{10}, a=3,4;\\
	(a-2)(a+1)^8a^6,a\geq 5.\\
\end{cases}$$
  	\item[(iv)]
      For $n=5$ and $a\ge 3$,
  		$ex_\Pi (5,5,\binom{5}{2}a+4)=\Pi_{2,2}(a,5)=(a+1)^4a^6$.

\end{enumerate}
\end{thm}
\par

It is not difficult to verify that ${\Sigma}_{2,2}(a,5)=\binom{5}{2}a+4$, therefore, our results above imply that 
\eqref{conj} does not hold for any $a\ge 3$ when $n=6$.  This shows that for $(s,q)=(5,\binom{5}{2}a+4)$ and $d=2$, $n$ needs to be sufficiently large for \eqref{conj} to hold.

     The rest of the paper is organized as follows. In
Section 2, we introduce some more notation and preliminaries.
We give the proof of Theorem \ref{MS1} in Section 3.

    \section{Preliminaries}
    The following integral version of the AM-GM inequality was presented in \cite{DFT},  we will use it frequently.
	\begin{lem}\label{gg}\cite{DFT}
	Let $ a,n\in [0,n]$, and let $\omega_1,...,\omega_n$ be non-negative integers with $\sum_{i=1}^{n} \omega_i=an+t$.
	Then the following hold:\\
	(i)  $ \Pi_{i=1} ^n \omega_i \leq a^{n-t}(a+1)^t$;\\
	(ii)  if $ t\leq n-2$ and $\omega_1=a-1$
	then $\Pi_{i=1} ^t \omega_i \leq (a-1)a^{n-t-2}(a+1)^{t+1}$.
	\end{lem}

Considering the maximum $P(G)$ among graphs in $\mathcal{T}_{2,2}(a,n)$, we obtain the following lemma.
	\begin{lem}\label{MS}
		Let $a\in [3,+\infty)$, $n\in [5,+\infty )$,
		for the graph $G\in \mathcal{T}_{2,2}(a,n)$, let $\cup_{i=0}^{1}V_i$ be the canonical partition of $G$. Set $|V_0|=x$, then $P(G) $ is maximum among all the graphs in $\mathcal{T}_{2,2}(a,n)$ if and only if
	$$ n\in \left((x-1)\frac{\ln(1-\frac{3}{a+1})}{\ln(1-\frac{1}{a+1})}+x,x \frac{\ln(1-\frac{3}{a+1})}{\ln(1-\frac{1}{a+1})}+x+1\right]
.$$
		Moreover, $$\Pi_{2,2}(a,n)=(a-2)^{\binom{x}{2}}a^{\binom{n-x}{2}}(a+1)^{x(n-x)}$$
for $ n\in \left((x-1)\frac{\ln(1-\frac{3}{a+1})}{\ln(1-\frac{1}{a+1})}+x,x \frac{\ln(1-\frac{3}{a+1})}{\ln(1-\frac{1}{a+1})}+x+1\right]$.	\end{lem}
     \begin{proof}
     Let $G$ be an $n$-vertex graph from $\mathcal{T}_{2,2}	(a,n).$ Suppose that $V_0$ of $G$ has $x$ vertices, and $V_1$ has $n-x$ vertices.
     By simple calculation, when  $V_0$ has $x$ vertices, the product of the edge multiplicities is $(a-2)^{\binom{x}{2}}a^{\binom{n-x}{2}}(a+1)^{x(n-x)}$.
     Now we define the new graph $G^{'}$ to be a graph obtained from $G$ by moving a vertex from $V_1$ to $V_0$. Considering the changing of the product of the edge multiplicities by moving a vertex from $V_1$ to $V_0$, we have that $\frac{P(G)}{P(G^{'})}=\frac{a^{n-x-1}(a+1)^x}
     {(a-2)^x(a+1)^{n-x-1}}$,
     which means that the product of the edge multiplicities is increased when the quantity is less than one and otherwise decreased. \par
       If
       $n\in \left((x-1)\frac{\ln(1-\frac{3}{a+1})}{\ln(1-\frac{1}{a+1})}+x,x \frac{\ln(1-\frac{3}{a+1})}{\ln(1-\frac{1}{a+1})}+x+1\right],$
       then we find that $n$ satisfies the following inequalities:
        $$\begin{cases}
       	\frac{a^{n-x}(a+1)^{x-1}}{(a-2)^{x-1}(a+1)^{n-x}}<1,\\
       	\frac{a^{n-x-1}(a+1)^x}{(a-2)^x(a+1)^{n-x-1}}\geq 1, \\
       \end{cases}$$
   which means that the product of the edge multiplicities in $G$ is decreased whether moving a vertex from $V_1$ to $V_0$ or moving a vertex from $V_0$ to $V_1.$
   \par
        As $(a+1)^2\geq a(a-2)$ holds for any $a\geq3$, the inequality $\frac{a^{n-k}(a+1)^{k-1}}{(a-2)^{k-1}(a+1)^{n-k}}
        \leq\frac{a^{n-k-1}(a+1)^k}{(a-2)^k(a+1)^{n-k-1}}$ established when $k$ is a positive integer. Therefore, $n$ satisfies:
       \begin{equation}\label{ms}
       	\begin{cases}
       		\frac{a^{n-2}(a+1)}{(a-2)(a+1)^{n-2}}<1,\\
       		\frac{a^{n-3}(a+1)^2}{(a-2)^2(a+1)^{n-3}}<1,\\
       		...\\
       		\frac{a^{n-x}(a+1)^{x-1}}{(a-2)^{x-1}(a+1)^{n-x}}<1,\\
       		\frac{a^{n-x-1}(a+1)^x}{(a-2)^x(a+1)^{n-x-1}}\geq 1, \\
       		...\\
       		\frac{a^{1}(a+1)^{n-2}}{(a-2)^{n-2}(a+1)^{1}}\geq 1 ,\\
       		\frac{a^{0}(a+1)^{n-1}}{(a-2)^{n-1}(a+1)^{0}}\geq 1. \\
       	\end{cases}
       \end{equation}

       This system of inequalities means that the quantity of the product of the edge multiplicities is maximum when $|V_0|=x.$
       Thus we have proved the sufficiency of Lemma \ref{MS}. Now we  prove the necessity of Lemma \ref{MS}.\par

       For the necessity of Lemma \ref{MS}, if $P(G)$ is maximum when $|V_0|=x$, then both adding a vertex to $V_0$ and taking out a vertex from $V_0$ will decrease the quantity of the product of the edge multiplicities, which means that:

       $$\begin{cases}
       	\frac{a^{n-x}(a+1)^{x-1}}{(a-2)^{x-1}(a+1)^{n-x}}<1,\\
       	\frac{a^{n-x-1}(a+1)^x}{(a-2)^x(a+1)^{n-x-1}}\geq 1, \\
       \end{cases}$$
      Solving the above set of inequalities, we have that:
       $$n\in \left((x-1)\frac{\ln(1-\frac{3}{a+1})}{\ln(1-\frac{1}{a+1})}+x,x \frac{\ln(1-\frac{3}{a+1})}{\ln(1-\frac{1}{a+1})}+x+1\right].$$

        As a result, when $ n\in \left((x-1)\frac{\ln(1-\frac{3}{a+1})}{\ln(1-\frac{1}{a+1})}+x,x \frac{\ln(1-\frac{3}{a+1})}{\ln(1-\frac{1}{a+1})}+x+1\right]$, we have $\Pi_{2,2}(a,n)=(a-2)^{\binom{x}{2}}a^{\binom{n-x}{2}}(a+1)^{x(n-x)}$.

       \end{proof}

           \begin{lem}\label{l2}
       	Let $G$ be an $n$-vertex graph from $\mathcal{T}_{2,2}(a,n)$, and let $\cup_{i=0}^{1}V_i$ be the canonical partition of $G$. Then among all the graphs in $\mathcal{T}_{2,2}(a,n)$, the product of the edge multiplicities will be maximum when $|V_0|$=2 if and
       	only if
       	$$ n\in \left(\frac{\ln(1-\frac{3}{a+1})}{\ln(1-\frac{1}{a+1})}+2,
       2\frac{\ln(1-\frac{3}{a+1})}{\ln(1-\frac{1}{a+1})}+3\right].$$
       \end{lem}
       \begin{proof}
       	By substituting $x=2$ into the system of the inequalities  (\ref{ms}), we get $$ n\in \left(\frac{\ln(1-\frac{3}{a+1})}{\ln(1-\frac{1}{a+1})}+2,
       2\frac{\ln(1-\frac{3}{a+1})}{\ln(1-\frac{1}{a+1})}+3\right], $$ which proves the lemma.
        \end{proof}

       	If we let
       $F(a)=\frac{\ln(1-\frac{3}{a+1})}{\ln(1-\frac{1}{a+1})},a\geq 3$, we find the function is monotone decreasing for $a$. Therefore, the maximum of $ F(a)$ is $F(3)=4.82$. Also, it is not difficult to obtain that $\lim_{a\to +\infty} F(a)=3 $.

\section{Proof of Theorem \ref{MS1}}
\subsection{The case when \bf{$n\geq 7$} }
  For $n\geq 7$, let $G$ be a graph from $\mathcal{T}_{2,2}(a,n)$ and $n$ be a positive integer. Let $\cup_{i=0}^{1}V_i$ be the canonical partition of $G$. Suppose $P(G)$ is maximum when $|V_0|=x$, among all the graphs in $\mathcal{T}_{2,2}(a,n)$. By Lemma \ref{ms}, we know:
\begin{equation}\label{qq}
	n\in \left((x-1)\frac{\ln(1-\frac{3}{a+1})}{\ln(1-\frac{1}{a+1})}+x,x \frac{\ln(1-\frac{3}{a+1})}{\ln(1-\frac{1}{a+1})}+x+1\right].	
\end{equation}
Moreover, $\Pi_{2,2}(a,n)=(a-2)^{\binom{x}{2}}a^{\binom{n-x}{2}}(a+1)^{x(n-x)}$.
\par
Let $G$ be a graph from $\mathcal{F}(n,5,\binom{5}{2}a+4)$. By averaging over all 5-sets, we obtain that the number of edges of $G$ satisfying that
$$
e(G)\leq\left \lfloor\frac{\binom{n}{5}}{\binom{n-2}{3}}
\left(\binom{5}{2}a+4\right)
\right\rfloor
=\frac{n(n-1)}{2}a+\left\lfloor\frac{n(n-1)}
{5}\right\rfloor
.$$
By Lemma \ref{gg} (i), we have
$$
P(G)\leq a^{\frac{n(n-1)}{2}-\lfloor\frac{n(n-1)}{5}\rfloor}(a+1)^{\lfloor\frac{n(n-1)}{5}\rfloor}.
$$\par
As a result, we have
\begin{align*}
	\Pi_{2,2}(a,n)&=(a-2)^{\binom{x}{2}}a^{\binom{n-x}{2}}(a+1)^{x(n-x)}\\
	&\leq ex_\Pi (n,5,\binom{5}{2}a+4)\\
	&\leq a^{\frac{n(n-1)}{2}-\lfloor\frac{n(n-1)}{5}\rfloor}(a+1)^{\lfloor\frac{n(n-1)}{5}\rfloor}.
\end{align*}

We will prove
\begin{equation}\label{e2}
	\lim_{a\to +\infty}\frac{a^{\frac{n(n-1)}{2}-\lfloor
			\frac{n(n-1)}{5}\rfloor}(a+1)^{\lfloor
			\frac{n(n-1)}{5}\rfloor}}{(a-2)^{\binom{x}{2}}
		a^{\binom{n-x}{2}}(a+1)^{x(n-x)}}=1.
\end{equation}
Note that the case (i) in Theorem \ref{MS1} follows if  the equality \eqref{e2} holds.
Let
$$\begin{cases}
	h_1(a)=a^{\frac{n(n-1)}{2}-\lfloor\frac{n(n-1)}{5}\rfloor}(a+1)^{\lfloor\frac{n(n-1)}{5}\rfloor},\\
	h_2(a)=(a-2)^{\binom{x}{2}}a^{\binom{n-x}{2}}(a+1)^{x(n-x)}.
\end{cases} $$
Expanding $h_1(a)$ and $ h_2(a)$, we have $$h_1(a)=a^{\frac{n(n-1)}{2}}+...+a^{\frac{n(n-1)}{2}-\lfloor\frac{n(n-1)}{5}\rfloor},$$ which is a polynomial on variable $a$ with the highest degree  $\frac{n(n-1)}{2}$, and the coefficient of $a^{\frac{n(n-1)}{2}}$ is 1. Also,
$$h_2(a)=a^{\frac{n(n-1)}{2}}+...+(-2)^{\frac{x(x-1)}{2}}a^{\frac{(n-x)(n-x-1)}{2}},$$ which is a polynomial on variable $a$ with the highest degree  $\frac{n(n-1)}{2}$, and the coefficient of $a^{\frac{n(n-1)}{2}}$ is 1. It follows that
$\lim_{a\to +\infty}\frac{h_1(a)}{h_2(a)}=1.$
\subsection{The case when \bf $n=7$ }
For $n= 7$, note that  $$ 7\in \left(\frac{\ln(1-\frac{3}{a+1})}{\ln(1-\frac{1}{a+1})}+2,
2\frac{\ln(1-\frac{3}{a+1})}{\ln(1-\frac{1}{a+1})}+3\right].$$
Applying  Lemma \ref{l2} for graph $H\in\mathcal{T}_{2,2}(a,7)$, if $H$ maximizes the product of the edge multiplicities, then $V_0$ of $H$ consists of two vertices.
Therefore, $\Pi_{2,2}(a,7)=(a-2)a^{10}(a+1)^{10}$. And $$ex_\Pi (7,5,\binom{5}{2}a+4)\ge \Pi_{2,2}(a,7)=(a-2)a^{10}(a+1)^{10}. $$
\par
Let $G$ be a product-extremal graph in
$\mathcal{F}(7,5,
\binom{5}{2}a+4)$.
By averaging over all 5-sets, we see that
$$e(G)\leq \left\lfloor\frac{\binom{7}{5}}{\binom {5}{3}}\left(\binom{5}{2} a+4\right)\right\rfloor= 21a+\left\lfloor\frac{42}{5}\right\rfloor=21a+8.$$
If $e(G)\leq 21a+7$, then by Lemma \ref{gg} (i),
$$P(G)\leq a^{14}(a+1)^7.$$
\par
Suppose now that
\begin{equation}\label{eG}
	e(G)=21a+8.
\end{equation}
\par
If $G$ contains at least one edge of multiplicity at most $a-1$, then by Lemma \ref{gg} (ii), $P(G)\leq (a-1)a^{11}(a+1)^9$. The equation holds if and only if there are one edge with multiplicity $a-1$ and 11 edges with multiplicity $a$ and 9 edges with multiplicity $a+1$.\par

On the other hand, suppose all the edges of $G$ have multiplicity at least $a$. Since $G\in\mathcal{F}(7,5,\binom{5}{2}a+4)$, the edges in $G$ have multiplicity between $a$ and $a+4$. We need to  consider the various values of the edge multiplicity of $G$, beginning with the easiest cases. \par
Case  \uppercase\expandafter{\romannumeral1}: Every edge of $G$ has multiplicity either $a$ or $a+1$. \par
From \eqref{eG}, we obtain that	$G$ has exactly 8 edges of multiplicity $a+1$.  We denote by $G^{a+1}$  the graph spanned by edges of multiplicity $a+1$ of $G$.   Then we  claim that there is a cycle on 4 vertices in $G^{a+1}$.
\begin{claim}\label{qq}
	There must have a $C_4$  in $G^{a+1}$.
\end{claim}
\begin{proof}
	Above all, there must be a path of length 3 in $G^{a+1}$. Indeed, since $G^{a+1}$ has 7 vertices and 8 edges, it must have a path of length 2. If there is not a path of length 3, then $G^{a+1}$ must have as many as possible paths of length 2 to cover the 8 edges in $G^{a+1}$. As there are 7 vertices,  there must be a vertex adjacent to the midpoint of only one path of length 2.  Then there are at most 5 edges in $G^{a+1}$, a contradiction. So $G^{a+1}$ must have a path of length 3.
	\par
	Suppose there is not a $C_4$ in $G^{a+1}$, we want to get a contradiction. We denote by $P$ a path of length 3 in $G^{a+1}$, and its vertices set is $\left\{1,2,3,4 \right\}$, such that $i$ and $i+1$ are adjacent in $P$ for $1\le i\le 3$. Then the endpoints of $P$ are 1 and 4 and the remained vertices are  $5,6,7$.
	 \par
	Since $G^{a+1}$ does not have a $C_4$, the vertices 1 and 4 are not adjacent in $G^{a+1}$.
	If $\{1,3\}$ and $\{2,4\}$ are edges in $G^{a+1}$, we find that the 5-vertex set $\left\{1,2,3,4,5\right\}$ spans edges with multiplicities at least $10a+5$, a contradiction with $G\in \mathcal{F}(7,5,\binom{2}{2}a+4)$.
	If one of $\{1,3\}$ and $\{2,4\}$ is an edge in $G^{a+1}$, then there exists an edge between two sets $\{1,2,3,4\}$ as $\left\{5,6,7 \right\}$ can span at most 3 edges in $G^{a+1}$. Hence we can obtain a 5-set containing vertices $1,2,3,4$, which spans edges with multiplicities summation at least $10a+5$.
	 \par
	
	The statement above means that  the vertices of $P$ spanned no edges in $G^{a+1}$ except the edges of the path. And we claim that for any  vertex $u\in \left\lbrace 5,6,7\right\rbrace $, it can send at most one edge into $P$  in $G^{a+1}$. Otherwise $\{u,1,2,3,4\}$ is a 5-vertex set which spans edges with multiplicities summation at least $10a+5$, a contradiction with  $G\in \mathcal{F}(7,5,\binom{5}{2}a+4)$.\par
	Note that there are 8 edges in $G^{a+1}$ and any  vertex in $\left\lbrace 5,6,7\right\rbrace$ can send at most one edge into $P$  in $G^{a+1}$ and $\left\lbrace 5,6,7\right\rbrace$ can span at most 3 edges in $G^{a+1}$, then we find that either  $\left\lbrace 5,6,7\right\rbrace$ span 3 edges or  $\left\lbrace 5,6,7\right\rbrace$  span 2 edges.\par
	Case a: If  $\left\lbrace 5,6,7\right\rbrace$ span 2 edges, then all vertices in  $\left\lbrace 5,6,7\right\rbrace$ must send one edge into $P$ in $G^{a+1}$. Consequently, there must either exist a $C_4$ or a 5-set which spans edges with multiplicities summation at least $10a+5$ , a contradiction.\par
	Case b: If  $\left\lbrace 5,6,7\right\rbrace$ span 3 edges,  without loss of generality we suppose that the vertex 5 sends no edges into $P$. There are two cases need to be consider. If there are two vertices in $\left\lbrace 5,6,7\right\rbrace$ are adjacent to the same vertex (without loss of generality we let it be $1$) in $\left\lbrace 1,2,3,4\right\rbrace $, then we find the two vertices with $\left\lbrace 1,2,3\right\rbrace $ form a 5-set which spans edges with multiplicities summation at least $10a+5$.
	Otherwise, without loss of generality we suppose the vertex 6 is adjacent to 1 and the vertex 7 is adjacent to 4 in $G^{a+1}$. Now we find a 5-set $\left\lbrace 1,4,5,6,7\right\rbrace $ which spans edges with multiplicities summation at least $10a+5$.
	Both case make a contradiction with  $G\in \mathcal{F}(7,5,\binom{2}{2}a+4)$.
	Thus, there is a $C_4$ in $G^{a+1}$.
\end{proof}\par

Let $C$ be a cycle of length 4 in $G^{a+1}$, and its vertices set is $\left\{1,2,3,4 \right\}$, such that $i$ and $i+1$ are adjacent in $C$ for $1\le i\le 3$ and 1 is adjacent to 4. As $G\in \mathcal{F}(7,5,
\binom{5}{2}a+4)$, the spanning graph of $\left\{1,2,3,4 \right\}$ in $G^{a+1}$ is only the $C$ and $\left\{5,6,7 \right\}$ send no edges into $C$ in $G^{a+1}$. However, $\left\{5,6,7 \right\}$ can span at most 3 edges with multiplicity $a+1$, which means there are at most 7 edges with multiplicity $a+1$, a contradiction. As a result, edges in $G$ can not only have multiplicity either $a$ or $a+1$.\\

Case \uppercase\expandafter{\romannumeral2} : $G$ contains an edge $e_0$ with multiplicity $a+2$.\par

Since any 5 vertices can span edges with multiplicities at most $10a+4$, then the multiplicity of other edges are one of $a,a+1,a+2$. We consider the following three subcases.

Subcase 1: Suppose no edge with multiplicity $a+1$ or $a+2$ is incident with $e_0$. Then
\begin{align*}
	P(G)&\leq \omega(e_0)a^{10}ex_\Pi (5,5,\binom{5}{2}a+4)\\
	&=(a+2)a^{10}(a+1)^4a^6\\
	&=a^{16}(a+1)^4(a+2).
\end{align*}

Subcase 2: Suppose there is some vertex $v$ sending an edge of multiplicity $a+2$ to one of  endpoints of $e_{0}$ . Then the vertex is unique, it sends exactly one such edge into $e_0$, and every other edges must have multiplicity $a$ as any 5 points can span edges with multiplicities at most $10a+4$. Thus $P(G)\leq (a+2)^2a^{19}$.

Subcase 3: Suppose there is some vertex $v$ sending an edge $e_1$ of multiplicity $a+1$ to one of the endpoints of $e_{0}$. Let the endpoints of $e_0$ be $1,2$ and $v$ be $3$. Without losing generality we suppose $2$ is adjacent to  $3$.  Denote by $P$ the path with vertices 1, 2, 3. \par
If $1$ and $3$ are adjacent by an edge $e_2$ with multiplicity $a+1$, then all other edges must have multiplicity $a$. Otherwise we suppose the vertex 4 sends an edge with multiplicity more than $a$, we find $\left\{1,2,3,4\right\}$ with any vertex in $\left\{5,6,7\right\}$ form a 5-vertex set which spans edges with multiplicities summation at least $10a+5$. A contradiction with $G\in \mathcal{F}(7,5,\binom{5}{2}a+4)$. Then $P(G)\leq (a+2)(a+1)^2a^{18}$. \par

If $1$ and $3$  are adjacent by an edge $e_3$ with multiplicity $a$.
Then there is at most one vertex in $\left\{4,5,6,
7\right\}$ which can send an edge to $P$ with multiplicity at most $a+1$.

\begin{itemize}
	\item If there is a vertex  $4$ sending an edge $e_4$ with multiplicity $a+1$ to $P$, then the rest edges except $\left\lbrace e_0,e_1,e_3,e_4 \right\rbrace$ in $G[ \left\{1,2,3,4\right\} ]$ must have multiplicity $a$ and the rest vertices $5,6,7$ only can send edge with multiplicity $a$ into  $G[\left\{1,2,3,4\right\}]$. Now we consider the edges in  $G[\left\{5,6,7\right\}]$. If there is one edge $e_5$, without losing generality we assume that the endpoints of $e_5$ are $6,7$, and $e_5$ has multiplicity $a+2$. Then  $\left\{1,2,3,6,7\right\}$ is a 5-vertex set which spans edges with multiplicities summation exceed $10a+4$, a contradiction.
	If there is no edge having  multiplicity $a+2$, then $\left\{5,6,7\right\}$ can span at most 2 edges with multiplicity $a+1$. So $P(G)\leq a^{16}(a+1)^4(a+2)$.
	\item If there is no vertex in $\left\{
	4,5,6,7\right\}$ sends an edge with multiplicity $a+1$ to $P$. We claim that there is no edge with multiplicity $a+2$ in $G[\left\{5,6,7\right\}]$. Otherwise we suppose that $G[\left\{5,6,7\right\}]$ has an edge $e^{'}$ incident with $5,6$ with multiplicity $a+2$, then $\left\{1,2,3,5,6\right\}$ is a 5-vertex set which spans edges with multiplicities summation at least $10a+5$, a contradiction. Meanwhile we find there are at most 3 edges with multiplicity $a+1$, so $P(G)\leq a^{17}(a+1)^3(a+2)$.
\end{itemize}
In conclusion, if $G$ contains a edge $e_0$ with multiplicity $a+2$, then we have the following inequalities
	$$\begin{cases}
		P(G)\leq a^{16}(a+1)^4(a+2),\\
		P(G)\leq a^{19}(a+2)^2,\\
		P(G)\leq a^{18}(a+1)^2(a+2),\\
		P(G)\leq a^{16}(a+1)^4(a+2),\\
		P(G)\leq a^{17}(a+1)^3(a+2).\\
	\end{cases}$$
	Through comparing the right-hand side of these inequalities, we find $a^{16}(a+1)^4(a+2)$ is the maximum value. So if $G$ contains an edge $e_0$ with multiplicity $a+2$, then $P(G)\leq a^{16}(a+1)^4(a+2)$.

Case \uppercase\expandafter{\romannumeral3} : $G$ contains an edge $e_0$ of multiplicity $a+3$.\par

Since any 5 vertices can span edges with multiplicities at most $10a+4$,  every  edge except for $e_0$ has multiplicity either $a$ or $a+1$. Suppose there is some vertex $v$ sending an edge of multiplicity $a+1$ to one of the endpoints of $e_{0}$. As any 5 vertices can span edges with multiplicities at most $10a+4$, every other edge has multiplicity exactly $a$. Then $$P(G)\leq P(G[e_0\cup v])a^{19}=a^{19}(a+1)(a+3).$$
 On the other hand suppose there is no edge with multiplicity $a+1$ connect to $e_0$. We claim that there is no path of length 2 in $G^{a+1}$, otherwise there will be a 5-set which span edges with multiplicities at most $10a+4$. Hence there is at most two edge of multiplicity $a+1$. Then
 $$P(G)\leq a^{18}(a+1)^2(a+3).$$
\par
As
$$\frac{a^{18}(a+1)^2(a+3)}{a^{19}(a+1)(a+3)}>1~~\text{when}~~a\geq 3,$$
then if $G$ contains an edge $e_0$ of multiplicity $a+3$, we have
$$P(G)\leq a^{18}(a+1)^2(a+3).$$

Case \uppercase\expandafter{\romannumeral4} : $G$ contains an edge of multiplicity $a+4$. \par
Since any 5 vertices can span edges with multiplicities at most $10a+4$, every other edge has multiplicity exactly $a$, and $P(G)=(a+4)a^{20}$.\\ \par
Consider all the upper bounds for $P(G)$ we obtained above, we have
$$\begin{cases}
	P(G)\leq a^{14}(a+1)^7,\\
	P(G)\leq (a-1)a^{11}(a+1)^9,\\
	P(G)\leq a^{16}(a+1)^4(a+2),\\
	P(G)\leq a^{18}(a+1)^2(a+3),\\
	P(G)\leq a^{20}(a+4).\\
\end{cases}$$
Through comparison we finally get that $(a-1)a^{11}(a+1)^9$ is the maximum upper bound for $P(G)$. However, this bound is not achievable. In fact,  we obtained this bound $(a-1)a^{11}(a+1)^9$ when $G$ contains at least one edge of multiplicity at most $a-1$. The equality holds if and only if there is one edge, denoted by $e$ with multiplicity $a-1$ and 11 edges with multiplicity $a$ and 9 edges with multiplicity $a+1$.  By claim \ref{qq}, we know that  $G^{a+1}$ must have a $C_4$. Let the vertex set of $C_4$ be $\left\lbrace 1,2,3,4\right\rbrace$.\par
Suppose $\left\lbrace 1,2,3,4\right\rbrace$ contains both of the endpoints of $e$,  let the endpoints of $e$ be $1,3$. As $\left\lbrace 1,2,3,4\right\rbrace$ can span at most 5 edges with multiplicity  $a+1$ and $\left\lbrace 5,6,7\right\rbrace$ can span at most 3 edges with multiplicity  $a+1$, there is at least one vertex $u\in\left\lbrace 5,6,7\right\rbrace$ sending edge with multiplicity $a+1$ into $\left\lbrace 1,2,3,4\right\rbrace$. Actually, $\left\lbrace 1,2,3,4,u \right\rbrace$ is a 5-vertex set which spans edges with multiplicities summation at least $10a+5$, a contradiction.\par
Suppose $\left\lbrace 1,2,3,4\right\rbrace$ contains one of the endpoints of $e$, then let the endpoints of $e$ be $1,5$. As $\left\lbrace 1,2,3,4\right\rbrace$ can span only 4 edges with multiplicity  $a+1$ and $\left\lbrace 5,6,7\right\rbrace$ can span at most 3 edges with multiplicity  $a+1$. So there is at least one vertex in $\left\lbrace 5,6,7\right\rbrace$ sending edge with multiplicity $a+1$ into $\left\lbrace 1,2,3,4\right\rbrace$. If 6 or 7 sends edges with multiplicity $a+1$ into $\left\lbrace 1,2,3,4\right\rbrace$, then $\left\lbrace 1,2,3,4 \right\rbrace$ with 6 or 7 form a 5-vertex set which spans edges with multiplicities summation at least $10a+5$, a contradiction. Otherwise 5 must send exactly 2 edges with multiplicity $a+1$ into $\left\lbrace 1,2,3,4\right\rbrace$. Then $\left\lbrace 5,6,7 \right\rbrace$ with the endpoints of the two edges form a 5-vertex set which spans edges with multiplicities summation at least $10a+5$, a contradiction. \par
Suppose $\left\lbrace 1,2,3,4\right\rbrace$ contains no endpoints of $e$,  let the endpoints of $e_0$ be $5,6$. As $\left\lbrace 1,2,3,4\right\rbrace$ can span only 4 edges with multiplicity  $a+1$ and $\left\lbrace 5,6,7\right\rbrace$ can span at most 2 edges with multiplicity  $a+1$. So there is at least one vertex $u\in \left\lbrace 5,6,7\right\rbrace$ sending edge with multiplicity $a+1$ into $\left\lbrace 1,2,3,4\right\rbrace$. Actually, $\left\lbrace 1,2,3,4,u \right\rbrace$ is a 5-vertex set which spans edges with multiplicities summation at least $10a+5$, a contradiction.
 Therefore, we have that $P(G)$ is strictly less than  $(a-1)a^{11}(a+1)^9$. \par
As $\Pi_{2,2}(a,7)\in \mathcal{F}(7,5,\binom{5}{2}a+4)$, we have 
\begin{align*}
	\Pi_{2,2}(a,7)&=(a-2)(a+1)^{10}a^{10}\\
	&\leq ex_\Pi (7,5,\binom{5}{2}a+4)\\
	&<(a-1)a^{11}(a+1)^9,
\end{align*}
 and
  $$\lim_{a\to +\infty}\frac{(a-1)a^{11}(a+1)^9}{(a-2)(a+1)^{10}a^{10}}=1$$
  which proves the case (ii) in Theorem \ref{MS1}.

\par
\subsection{The case when \bf $n=6$ }
 For $n=6$, let $H$ be a product extremal graph in $\mathcal{F}(6,5,\binom{5}{2}a+4)$.
 By averaging over all 5-sets, we see that $$e(H)\leq \frac {\binom {6}{5}}{\binom{4}{3}}\left(\binom{5}{2}a+4\right) =15a+6.$$
 By Lemma \ref{gg} (i), $P(H)\leq a^9(a+1)^6.$ Therefore, $ex_\Pi (6,5,\binom{5}{2}a+4)\leq a^9(a+1)^6.$
 On the other hand, consider the 6-vertex multigraph $H^{'}$ whose edges  of multiplicity $a+1$ form a 6-cycle $C_6$, and all other edges are of multiplicity $a$. We have $P(H^{'})=a^9(a+1)^6$ and $H\in$ $\mathcal{F}(6,5,\binom{5}{2}a+4)$. Hence, $ex_\Pi (6,5,\binom{5}{2}a+4)\geq a^9(a+1)^6.$ In conclusion, we have  $ex_\Pi (6,5,\binom{5}{2}a+4)=a^9(a+1)^6.$
   \par
   Let
$G\in \mathcal{T}_{2,2}(a,6)$.
When $a=3,4$, we have $$ 6\in \left(1,\frac{\ln(1-\frac{3}{a+1})}{\ln(1-\frac{1}{a+1})}+2\right],$$ which means the quantity $P(G)$ is maximised when $V_0=\left\{1\right\} $ and $V_1=\left\{2,3,4,5,6\right\}$ by lemma \ref{MS}. Then $\mathcal{T}_{2,2}(a,6)=(a+1)^5a^{10}$ when $a=3,4$.

When $a\geq 5$, we have
$$6\in \left(\frac{\ln(1-\frac{3}{a+1})}{\ln(1-\frac{1}{a+1})}+2,
2\frac{\ln(1-\frac{3}{a+1})}{\ln(1-\frac{1}{a+1})}+3\right]
,$$ which means the quantity $P(G)$ is maximised when $V_0=\left\{1,2\right\}$ and $V_1=\left\{3,4,5,6\right\}$ by lemma \ref{MS}. Then  $\mathcal{T}_{2,2}(a,6)=(a-2)(a+1)^8a^6$ when $a\geq  5$.
So we have that
$$\Pi_{2,2}(a,6)=
\begin{cases}
	(a+1)^5a^{10}, a=3,4;\\
	(a-2)(a+1)^8a^6,a\geq 5.\\
\end{cases}$$
Comparing $\Pi_{2,2}(a,6)$ and $a^9(a+1)^6$, we find that $(a+1)^5a^{10}$ is strictly less than $a^9(a+1)^6$ when $a=3,4$, and $(a-2)(a+1)^8a^6$ is strictly less than $a^9(a+1)^6$ when $a\geq 5$. Hence
$$ex_\Pi (6,5,\binom{5}{2}a+4)=a^9(a+1)^6>\Pi_{2,2}(a,6),~~\text{when}~~a\geq 3.$$
\par
\subsection{The case when \bf  $n=5$ }
   For $n=5$, let $G$ be a product-extremal graph from $\mathcal{F}(5,5,\binom{5}{2}a+4)$. Then $e(G)\leq \binom{5}{2}a+4=10a+4$. By Lemma \ref{gg} (i), we have
         $P(G)\leq a^{(10-4)}(a+1)^4=a^6(a+1)^4$. Hence, $ex_\Pi (5,5,\binom{5}{2}a+4)\leq a^6(a+1)^4$.

        On the other hand, partitioning [5] into $V_0=\left\{1\right\} $ and $V_1=\left\{2,3,4,5\right\}$, we have that
        $$ (a+1)^4a^6 \leq \Pi_{2,2}(a,5) \leq ex_\Pi (5,5,\binom{5}{2}a+4)\leq a^6(a+1)^4.$$ Therefore, we have proved the case (iv) in Theorem \ref{MS1}.\\

\noindent {\bf Acknowledgments.}
Ran Gu was partially supported by
 National Natural Science Foundation of China (No. 11701143).

   \end{document}